\documentclass[11pt]{amsart}   
\usepackage{latexsym}
\usepackage{amsfonts}
\usepackage{amsmath,amsthm,amssymb}
\usepackage{fullpage}   

\usepackage{latexsym}
\usepackage{amsfonts}
\usepackage{amsmath,amsthm,amssymb}
\usepackage{mathabx}
\usepackage{tikz}
\usetikzlibrary{matrix}
\usepackage[colorlinks=true, linkcolor=blue, citecolor=magenta]{hyperref}
\usepackage{mathrsfs}
\usepackage{soul}
\usepackage[all]{xy}
\usepackage{arydshln}

\newcommand{\cal}{\mathcal}

\newcommand{\chop}{\dagger}

\def\epsilon{\varepsilon}
\def\phi{\varphi}

\def\hat{\widehat}

\newcommand{\supp}{\mbox{\rm Supp}}

\newcommand{\card}{\mbox{\rm card}}

\newcommand{\FN}{F_N}   

\newcommand{\R}{\mathbb R}
\newcommand{\Z}{\mathbb Z}

\newcommand{\N}{\mathbb N}

\def\strutdepth{\dp\strutbox}
\def \ss{\strut\vadjust{\kern-\strutdepth \sss}}
\def \sss{\vtop to \strutdepth{
\baselineskip\strutdepth\vss\llap{$\diamondsuit\;\;$}\null}}

\def\strutdepth{\dp\strutbox}
\def \sst{\strut\vadjust{\kern-\strutdepth \ssss}}
\def \ssss{\vtop to \strutdepth{
\baselineskip\strutdepth\vss\llap{$\spadesuit\;\;$}\null}}

\def\strutdepth{\dp\strutbox}
\def \ssh{\strut\vadjust{\kern-\strutdepth \sssh}}
\def \sssh{\vtop to \strutdepth{
\baselineskip\strutdepth\vss\llap{$\heartsuit\;\;$}\null}}

\def\qed{\hfill\rlap{$\sqcup$}$\sqcap$\par}
\def\bar{\overline}

\def\strutdepth{\dp\strutbox}
\def \ss{\strut\vadjust{\kern-\strutdepth \sss}}
\def \sss{\vtop to \strutdepth{
\baselineskip\strutdepth\vss\llap{$\diamondsuit\;\;$}\null}}

\def\strutdepth{\dp\strutbox}
\def \sst{\strut\vadjust{\kern-\strutdepth \ssss}}
\def \ssss{\vtop to \strutdepth{
\baselineskip\strutdepth\vss\llap{$\spadesuit\;\;$}\null}}

\def\qed{\hfill\rlap{$\sqcup$}$\sqcap$\par}

\def\¶{\partial}

\newtheorem{thm}{Theorem}[section]
\newtheorem{cor}[thm]{Corollary}
\newtheorem{lem}[thm]{Lemma}
\newtheorem{prop}[thm]{Proposition}

\newtheorem{conjecture}[thm]{Conjecture}
\theoremstyle{definition}
\newtheorem{defn}[thm]{Definition}

\newtheorem{rem}[thm]{Remark}

\newtheorem{defn-rem}[thm]{Definition-Remark}

\theoremstyle{remark}

\numberwithin{equation}{section}

\begin{document} 
 
\author[M.~Lustig]{Martin Lustig}
\address{\tt 
Aix Marseille Universit\'e, CNRS, Centrale Marseille, I2M UMR 7373,
13453  Marseille, France
}
\email{\tt Martin.Lustig@univ-amu.fr}

\title[Complexity function of subshifts and morphisms]
{How do 
topological entropy and factor complexity behave under
monoid morphisms 
and free group basis changes\,?}
 
\begin{abstract} 
For any non-erasing free monoid morphism $\sigma: \cal A^* \to \cal B^*$, and for any subshift $X \subset \cal A^\Z$ and its image subshift $Y = \sigma(X) \subset \cal B^\Z$, 
the associated complexity functions $p_X$ and $p_Y$ 
are shown to satisfy:
there exist constants $c, d, C > 0$ such that
$$c \cdot p_X(d \cdot n) \,\, \leq \,\, p_Y(n) \,\, \leq \,\, C \cdot p_X(n)$$
holds for all sufficiently large integers $n \in \N$, provided that $\sigma$ is recognizable in $X$.
If $\sigma$ is in addition letter-to-letter, then $p_Y$ belongs to $\Theta(p_X)$ (and conversely). Otherwise, however, there are examples where $p_X$ is not in  $\cal O(p_Y)$.

It follows that in general the value $h_X$ of the topological entropy of $X$ is not preserved when applying a morphism $\sigma$ to $X$, even if $\sigma$ is recognizable in $X$.

As a consequence, there is no meaningful way to define the topological entropy of a current on a free group $F_N$; only the distinction of currents $\mu$ with topological entropy $h_{\tiny\supp(\mu)} = 0$ and $h_{\tiny\supp(\mu)} > 0$ is well defined.
\end{abstract}

 
 
\maketitle 

\section{Introduction}

Let $\cal A$ be a non-empty finite set, called an {\em alphabet}. We denote by $\cal A^*$ the free monoid over $\cal A\,$; its elements $w = x_1 \ldots x_n$ (with all $x_i \in \cal A$) are called {\em words in $\cal A$}, and $|w| = n$ is the {\em length} of $w$. In analogy we call an element ${\bf x}$ of the {\em shift space} $\cal A^\Z$ a {\em biinfinite word} in $\cal A$ and write it as 
\begin{equation}
\label{eq1.1-}
{\bf x} = \ldots x_{n-1} x_n x_{n+1} \ldots \quad\text{(with $x_n \in \cal A$ for any $n \in \Z$).}
\end{equation}

A non-empty subset $X \subset \cal A^\Z$ is called a {\em subshift
(over $\cal A$)} 
if it is closed with respect to the product topology on $\cal A^\Z$ 
(for the discrete topology on $\cal A$), 
and if it is invariant under the shift operator $T_\cal A$ (which acts on $\cal A^\Z$ through decreasing by 1 
all indices in any biinfinite word ${\bf x}$ 
as in (\ref{eq1.1-}). 

The set of {\em factors} $x_\ell \ldots x_m \in \cal A^*$ of any ${\bf x} \in X$ is called the {\em language} of the subshift $X$ and is denoted by $\cal L(X)$. Conversely, for any infinite set $\cal L \subset \cal A^*$ we denote by $X(\cal L)$ the subshift {\em generated by} $\cal L$, which is defined as the set of all ${\bf x} \in \cal A^\Z$ for which every factor is also a factor of some $w \in \cal L$. 

We denote by $\Sigma(\cal A)$ the set of all subhifts $X \subset \cal A^\Z$, and by $\Lambda(\cal A)$ the set of infinite subsets $\cal L \subset \cal A^*$ 
that are {\em factorial} (i.e. every factor of some $w \in \cal L$ also belongs to $\cal L$) and {\em bi-extendable} (i.e. every $w \in \cal L$ occurs also as factor in some $u \in \cal L$, but neither as prefix nor as suffix of $u$). Then the maps $X \mapsto \cal L(X)$ and $\cal L \mapsto X(\cal L)$ defines a well known canonical bijection:
$$\Sigma(\cal A) \,\,\, \longleftrightarrow \,\,\, \Lambda(\cal A)$$

The issuing double-nature of the basic objects in symbolic dynamics is on one hand the deep reason for the astonishing richness of this beautiful mathematical domain; on the other hand it is also the source of certain basic ``misunderstandings'', some of which 
are even up to date not completely straightened out (see for instance section 2 of  \cite{BHL2.8-I}). 

One of these problems comes from the notion of a ``morphism'', which has indeed two conflicting natural interpretations, for any second alphabet $\cal B$ and any second subshift $Y \subset \cal B^\Z$:  
\begin{enumerate}
\item
If we think of $X \in \Sigma(\cal A)$ as topological dynamical system $(X, T_\cal A)$, then a morphism $(X, T_\cal A) \mapsto (Y, T_\cal B)$ is given by a continuous map $\theta: X \to Y$ which commutes with the shift operators: 
$$T_\cal B \circ \theta = \theta \circ T_\cal A$$

\item
If, instead, we consider primarily the subshift language $\cal L(X) \subset \cal A^*$, then any non-erasing monoid morphism $\sigma: \cal A^* \to \cal B^*$ defines an infinite image set $\sigma(\cal L(X)) \subset \cal B^*$ which in turns gives rise to the {\em image subshift} $$\sigma(X) := X(\sigma(\cal L(X))\, .$$
\end{enumerate}

Recall here that any monoid morphisms $\sigma: \cal A^* \to \cal B^*$ is determined by the choice of the finite family of elements $\sigma(a_i) \in \cal B^*$ for any $a_i \in \cal A$, and that conversely, any such choice defines a monoid morphism.  The morphism $\sigma$ is {\em non-erasing} if none of the $\sigma(a_i)$ is the empty word. The morphism $\sigma$ is said to be {\em recognizable in $X$} if, roughly speaking, any biinfinite word in $\sigma(X)$ can be lifted via $\sigma$ to a biinfinite word in $X$ in at most one way. The precise definition is bit tedious and delayed here until section \ref{sec:3} (see Definition \ref{3.0a}).

\begin{rem}
\label{2-notions}
There is a natural intersection of the two notions (1) and (2) above, given by {\em letter-to-letter morphisms} $\sigma: \cal A^* \to \cal B^*$, which are monoid morphisms subject to the condition that $|\sigma(a_i)| = 1$ for any letter $a_i \in \cal A$. Indeed, as explained in the subsequent paragraph, any ``morphism'' in the sense of (1) above can canonically be reduced to a letter-to-letter morphisms in the meaning of (2). This is the reason why we adopt in this paper the wider interpretation (2) above whenever the term ``morphism'' will be used in the sequel.

A classical argument based on the celebrated Curtis-Hedlund-Lyndon theorem shows that any continuous map $\theta$ as in (1) above is induced by a letter-to-letter morphism $\sigma: \cal A_n^* \to \cal B^*$ with $Y = \sigma(X')$ and $X' = \rho_{n,k}^{-1}(X)$, where $\cal A_n = \{w \in \cal A^* \mid |w| = n\}$ and (for any $1\leq k\leq n$) the map $\rho_{n,k}:\cal A_n^* \to \cal A^* \, , w = x_1 \ldots x_n \mapsto x_k$ is a ``sliding block code'' morphism,
which canonically 
induces a homeomorphism $(X', T_{\cal A_n}) \mapsto (X, T_\cal A)$ 
for any integers $n$ and $k$ as above.
\end{rem}

The number of factors $x_\ell \ldots x_m$ of length $n := m-\ell$ of any ${\bf x} \in X$ as in (\ref{eq1.1-}) is denoted by $p_X(n)$; the issuing function
\begin{equation}
\label{eq1.2a}
p_X: \N \to \N \, , \,\, n \mapsto p_X(n)
\end{equation}
is called the {\em complexity function} (or {\em combinatorial complexity} or {\em factor complexity}) of the subshift $X$.

The complexity function $p_X$ has been investigated ever since symbolic dynamics has started out with the work of Morse and Hedlund in the 1930's. It is by now one of the most prominent tools in the study of subshifts; too many results are known to even start listing them here. A classification of $X$ according to the growth type of the monotonously growing function $p_X$ has turned out to be very fruitful, but many delicacies (for instance the potential discrepancy between $\limsup p_X(n)$ and  $\liminf p_X(n)$) come into play and still occupy the symbolic dynamics community until the very present. 

One reason of its importance is that the complexity function is a refinement of the {\em topological entropy} $h_X$ of a subshift $X$. This invariant, defined also in far more general contexts, turns out to be related to the complexity function by the following equality:
\begin{equation}
\label{eq1.3}
h_X \,\, = \,\, \lim_{n\to \infty}\frac{\log p_X(n)}{n}
\end{equation}
Grosso modo it seems fair to say that subshifts $X$ with entropy $h_X > 0$ are ``large''; for instance they do occur naturally in the context of regular languages. It is for the ``small'' subshifts $X$, i.e. with entropy $h_X = 0$, that the complexity function serves as finer measure for the seize of $X$. There is also a very interesting and not so well understood ``grey area'' where one has $h_X = 0$, but other invariants like the rank of $X$ or the number $e(X)$ of ergodic probability measures on $X$ indicate that $X$ behaves a lot more like what one knows from the positive entropy case, rather than for example from subshifts with linear complexity.

\medskip

We are now ready to state the main result of this note:

\begin{thm}
\label{1.1}
Let $\cal A$ and $\cal B$ be non-empty finite alphabets, and let $X \subset \cal A^\Z$ be a subshift. Let $\sigma: \cal A^* \to \cal B^*$ be a non-erasing morphism of free monoids, and let $Y:= \sigma(X)$ be the image subshift of $X$. Then the associated complexity functions $p_X$ and $p_Y$ satisfy the following:
\begin{enumerate}
\item
There exists a constant $C > 0$ such that 
$$
p_Y(n) \,\, \leq \,\, C \cdot p_X(n)
$$
for all 
$n \in \N$. Indeed, we can specify the constant to $C = \max\{|\sigma(a_i)| \mid a_i \in \cal A\}$.

\item
If $\sigma$ is recognizable in $X$ and letter-to-letter, then 
there exists a constant $c > 0$ such that
$$
c \cdot p_X(n) \,\, \leq \,\, p_Y(n)
$$
for all sufficiently large integers $n \in \N$.

\item
If $\sigma$ is recognizable in $X$ (but not necessarily letter-to-letter), then 
there exist constants $c > 0$ and $d > 0$ such that
$$
c \cdot p_X(d\cdot n) \,\, \leq \,\, p_Y(n)
$$
for all sufficiently large integers $n \in \N$.
\item
There exist subshifts $X \subset \cal A^\Z$ and morphisms $\sigma$ which are recognizable in $X$ such that for any constant $c > 0$ one has
$$
c \cdot p_X(n) \,\, > \,\, p_Y(n)
$$
for infinitely many 
integers $n \in \N$.
\end{enumerate}
\end{thm}

Recall 
that 
classical analysis symbols due to Landau and others have given rise to the following (here slightly modernized) terminology:
for any functions $f: \N \to \R_{>0}$ and $g: \N \to \R_{>0}$ we write $f \in \Theta(g)$ if an only if there exist constants $c > 0$ and  $C > 0$ such that
$$c \cdot g(n) \,\, \leq\,\,  f(n) \,\, \leq\,\,  C \cdot g(n)$$
holds for all sufficiently large $n \in \N$. This is equivalent to stating
$0 <  \underset{n \to\infty}{\liminf} \frac{f(n)}{g(n)}  \leq  \underset{n \to\infty}{\limsup} \frac{f(n)}{g(n)}  <  \infty \,$. 
It follows directly that the property $f \in \Theta(g)$ defines an equivalence relation on the set $\cal N := \R_{>0}^\N$ of all functions from $\N$ to $\R_{>0}$, where the equivalence class of any $f \in \cal N$ is precisely the set $\Theta(f)$. 

[Aside: Note that the classical equivalence relation $\sim$ (as used in calculus for real functions) is stronger, so that each of its equivalence classes in $\cal N$ is contained in some $\Theta(f)$.]

\medskip

We can thus summarize 
part of Theorem \ref{1.1} as follows:

\begin{cor}
\label{1.3}
Let $X \subset \cal A^\Z$ be a subshift, and let $\sigma: \cal A^* \to \cal B^*$ be a non-erasing morphism which is recognizable in $X$. Then one has:
\begin{enumerate}
\item
If $\sigma$ is letter-to-letter, then $p_{\sigma(X)} \in \Theta(p_X)$.

\item
If $\sigma$ is not letter-to-letter, then $p_{\sigma(X)} \in \cal O(p_X)$, but in general one has $p_{\sigma(X)} \notin \Theta(p_X)$.
\qed
\end{enumerate}
\end{cor}

The main part of Theorem \ref{1.1}, namely statement (3), will be derived below in subsection \ref{sec:3.2}. The proof is based on previous work of the author on recognizable morphisms, quoted and explained in subsection \ref{sec:3.1}. In section \ref{sec:4} we observe that the inequality 
from statement (2) 
follows indeed already from what has been done in section \ref{sec:3}. We then proceed to give a concrete counter-example to this stronger inequality when the ``letter-to-letter'' hypothesis is missing, thus proving statement (4). The morphism used there is simply given by a decomposition of every alphabet letter $a_i$ as product $a_i = a_i^{-} a_i^{+}$, and for $X$ we can take the full shift $\cal A^\Z$.
For completeness and as warm-up we give in subsection \ref{sec:2} a quick proof of part (1) of Theorem \ref{1.1}, although this is certainly folk knowledge among the experts. 
In section \ref{sec:5} we discuss some consequences (for instance for the 
topological entropy) and also draw the connection to algebraic laminations and currents for free groups $F_N$ of finite rank $N \geq 2$.
More precisely, we show (for the terminology see section \ref{sec:5}):

\begin{prop}
\label{5.2}
Let $F_N$ be a free group of finite rank $N \geq 2$, and let $L$ be an algebraic lamination in $F_N$. Let $X_\pm = X_\cal A(L) \subset \cal A_\pm^\Z$ and $Y_\pm = X_\cal B(L) \subset \cal B_\pm^\Z$ be the subshifts associated to $L$ via choices of bases $\cal A$ and $\cal B$ of $F_N$ respectively. Then the complexity functions $p_{X_\pm}$ and $p_{Y_\pm}$ satisfy: 
\begin{enumerate}
\item
There exist constants $c, d, C, D > 0$ such that one has
$$c \cdot p_{X_\pm}(\lfloor d \cdot n \rfloor) \,\, \leq \,\, p_{Y_\pm}(n) \,\, \leq \,\, C \cdot p_{X_\pm}(D\cdot n)$$
for any sufficiently large integer $n \geq 0$. 
(Here $\lfloor d\cdot n \rfloor$ denotes as usual the largest integer $m \leq d\cdot n$.)

\item
In general the statement (1) fails if in addition one imposes $c = d = 1$. In particular, there are bases $\cal A$ and $\cal B$ and algebraic laminations $L$ in $F_N$ such that
$$p_{Y_\pm} \notin \Theta(p_{X_\pm})\, .$$
\item
There is no well defined notion of a ``topological entropy'' $h_L \in \R$ for an algebraic lamination $L$ in $F_N$. However, the statements
$$h_L \,\, = \,\, 0 \quad\text{or} \quad h_L \,\, > \,\, 0$$
are well defined, since one has $h_{X_\pm} = 0 \Longleftrightarrow h_{Y_\pm} = 0$.
\end{enumerate}
\end{prop}

\bigskip

\noindent
{\em Motivation for this paper:} 
There are two reasons why the author addresses the 
topics 
of this paper (which aren't really his main expertise):

\smallskip
\noindent
(1)  
The author freely admits that for quite some time he thought that ``$\sigma$ is recognizable in $X$'' ought to imply $p_{\sigma(X)} \in \Theta(p_X)$. Indeed, a detailed ``proof'' only failed in its last inequality, and only then the search for a counter-example started. We suspect that we are not the only one who may fall into this trap; one purpose of this scholarly note is 
to warn our colleagues and to straighten out the whole issue.

\smallskip
\noindent
(2)  
Complexity functions of subshifts belong to the very useful toolbox that has been extensively developed over the years in symbolic dynamics, while being essentially ignored by the cousin community in geometric group theory (where in particular my fellow researchers on automorphisms of free groups and Outer space ought to know better). I hence made a purposeful effort to make this note accessible for anybody with only a sketchy background in symbolic dynamics, probably at the expense of boring occasionally the experts from this field.

\medskip

\noindent
{\em Acknowledgements:} We would like to thank Nicolas B\'edaride, Arnaud Hilion and Gilbert Levitt for inspiring conversations in the context of the subject treated in this note. 
We also would like to point the reader's attention to forthcoming work of Hilion-Levitt \cite{HL} which is in part related to the material of this note. In particular, a statement very close to part (1) of Proposition \ref{5.2} above is proved there by rather different methods.

\section{The upper and the lower bound for the image complexity function}
\label{sec:3}

\subsection{The upper bound}
\label{sec:2}

${}^{}$

\smallskip

Let $\sigma: \cal A^* \to \cal B^*$ be any morphism of free monoids over finite non-empty alphabets $\cal A$ and $\cal B$ respectively. 
It has become common use to denote by $||\sigma||$ the maximum and by $\langle \sigma \rangle$ the minimum of the lengths 
$|\sigma(a_i)|$ 
of the letter images $\sigma(a_i) \in \cal B^*$, for any $a_i \in \cal A$. We observe directly the following:

\begin{lem}
\label{X.1}
Let $\sigma: \cal A^* \to \cal B^*$ a non-erasing morphism, and let $X \subset \cal A^\Z$ be any subshift over $\cal A$, with image subshift $Y := \sigma(X) \subset \cal B^\Z$. 
Then the complexity functions $p_X$ and $p_Y$ satisfy
$$
p_Y(n) \,\, \leq \,\,  p_Y(\langle \sigma \rangle \cdot (n-1) +1) \,\,  \leq \,\,  ||\sigma|| \cdot p_X(n)$$
for any integer $n \in \N$.
\end{lem}

\begin{proof}
Any word $w$ in $\cal L(Y)$ of length $|w| = \langle \sigma \rangle \cdot (n-1) +1$ is a factor of $\sigma(u)$, for some 
word $u = x_1 \ldots x_n \in \cal L(X)$ of length $|u| = n$.
Furthermore we can assume that the first letter of $w$ occurs in $\sigma(x_1)$, or else we iteratively replace $u$ by $u' = x_2 \ldots x_n x_{n+1}$ any $x_{n+1} \in \cal A$. 

We also note that any factor $w$ of $\sigma(u)$ of a fixed given length is determined by the choice of $u$ and by the index of the first letter of $w$ as factor of $u$. By the assumption from the last paragraph we know that there are at most $|\sigma(x_1)| \leq ||\sigma||$ such possible indices. We derive directly that there are at most $||\sigma|| \cdot p_X(n)$ possible choices for the word $w$ with $|w| = \langle \sigma \rangle \cdot (n-1) +1$.

This proves the second of the two claimed inequalities; 
since 
$\langle \sigma \rangle \cdot (n-1) +1 = n + (\langle \sigma \rangle -1) \cdot (n-1)$ 
the first is just the well known monotony of the complexity function of any subshift.
\end{proof}

\subsection{Material from previous papers}
\label{sec:3.1}

${}^{}$

\smallskip

The notion of ``recognizability'' in symbolic dynamics has a long and fruitful history, and various versions of it have been in use over time. We refer here to the recently 
proposed 
specification
from \cite{BSTY19}, 
which has already been used in several subsequent papers and seems by now to be the established version in the field.

Before stating the formal definition, 
we have to be explicit about our convention how a non-erasing monoid morphism $\sigma: \cal A^* \to \cal B^*$ acts on an element ${\bf x} = \ldots x_{n-1} x_n x_{n+1} \ldots \in \cal A^\Z\,$: we first define two infinite half-words $y_1 y_2 \ldots := \sigma(x_1) \sigma(x_2) \ldots$ and $\ldots y_{-1} y_0 := \ldots \sigma(x_{-1}) \sigma(x_0)$ and then paste them together to obtain $\sigma({\bf x}) := \ldots y_{-1} y_0 y_1 \ldots$.

\begin{defn}
\label{3.0a}
Let $\sigma:\cal A^* \to \cal B^*$ be a non-erasing morphism, and let $X \subset \cal A^\Z$ be a subshift over $\cal A$. 
Then $\sigma$ is called {\em recognizable in $X$} 
if the following conclusion is true, for any 
biinfinite words ${\bf x}, {\bf x'} \in X \subset \cal A^\Z$ and ${\bf y} \in \cal B^\Z$, with ${\bf x} = \ldots x_{n-1} x_n x_{n+1} \ldots$ and ${\bf x'} = \ldots x'_{n-1} x'_n x'_{n+1} \ldots$.

Assume that ${\bf y} = T_\cal A^k(\sigma({\bf x}))$ and ${\bf y} = T_\cal A^\ell(\sigma({\bf x'}))$ for some integers $k, \ell$ which satisfy $0 \leq k \leq |\sigma(x_1)|-1$  and $ 0 \leq  \ell \leq |\sigma(x'_1)|-1$.
Then one has ${\bf x} = \bf x'$ and $k = \ell$.
\end{defn}

This combinatorial definition has been translated in Proposition 6.3 of \cite{BHL2.8-I} into a more conceptual property as follows:

\begin{prop}[\cite{BHL2.8-I}]
\label{3.0b}
Let $\sigma: \cal A^* \to \cal B^*$ be a non-erasing morphism of free monoids, and let $X \subset \cal A^\Z$ be any subshift over $\cal A$. Then $\sigma$ is recognizable in $X$ if and only if the morphism $\sigma$ is both, shift-orbit injective and shift-period preserving in $X$.
\qed
\end{prop}

Here ``shift-orbit injective in $X$'' means that the map induced by $\sigma$ on the shift-orbits of $X$ is injective, and ``shift-period preserving in $X$'' means that for any periodic orbit $\ldots w w w \ldots$ in $X$ the image of $w$ satisfies $\sigma(w) = u^r$ for some $u \in \cal B^*$ and $r \geq 2$ if and only if $w = v^r$ for some $v \in \cal A^*$. The proof of Proposition \ref{3.0b} is elementary, but to fill in all details is rather tedious. The author recommends not to do it as ``exercise'', but rather look up the version presented in \cite{BHL2.8-I}.

In \cite{BHL2.8-I} the following terminology has been introduced, following an earlier version from \cite{Lu-habil}.

\begin{defn}
\label{3.1}
(1)
For any word $w \in \cal A^*$ and any integer $r \geq 0$ we define $w\chop_r$ (``$w$ chop $r$'') to be the factor of $w$ obtained through deleting the prefix and the suffix of length $r$ from $w$. If $r \geq \frac{|w|}{2}$, then $w\chop_r$ is the empty word.

\smallskip
\noindent
(2)
For any morphism $\sigma: \cal A^* \to \cal B^*$ and any subshift $X \subset \cal A^\Z$ we say that an integer $r \geq 0$ is a {\em repetition bound for $\sigma$ on $X$} if for any two words $w, w' \in \cal L(X)$ with $\sigma(w) = \sigma(w')$ 
one has 
$w\chop_r = w'\chop_r$.
\end{defn}

\begin{rem}
\label{canonical-decomposition}
Any non-erasing morphism $\sigma: \cal A^* \to \cal B^*$ admits canonically a decomposition $\sigma = \alpha_\sigma \circ \pi_\sigma$, where the {\em subdivision morphism} $\pi_\sigma: \cal A^* \to \cal A_\sigma^*$ is defined by first setting $\cal A_\sigma := \{a_i(k) \mid a_i \in \cal A\, , \,\, 1 \leq k \leq |\sigma(a_i)| \}$ and then declaring $\pi_\sigma(a_i) := a_i(1) a_i(2) \ldots a_i(|\sigma(a_i)|)$. 
The morphism $\alpha_\sigma: \cal A_\sigma^* \to \cal B^*$ is letter-to-letter in that it maps any $a_i(k)$ to the $k$-th letter $y_k \in \cal B$ of the image word $\sigma(a_i) = y_1 \ldots y_{|\sigma(a_i)|} \in \cal B^*$.
\end{rem}

In \cite{BHL2.8-I} 
the above quoted Proposition \ref{3.0b} has been used to derive (fairly directly) 
the following:

\begin{prop}[\cite{BHL2.8-I}]
\label{3.2}
For any subshift $X \subset \cal A^\Z$ a non-erasing morphism $\sigma: \cal A^* \to \cal B^*$ is recognizable in $X$ if and only if there exists a repetition bound 
$r \geq 0$ 
for the induced letter-to-letter morphism $\alpha_\sigma$ on the image subshift $\pi_\sigma(X)$. 
\qed
\end{prop}

\medskip
\subsection{The lower bound}
\label{sec:3.2}

${}^{}$

\smallskip

For an arbitrary infinite set $\cal L \subset \cal A^*$ one has in general
$$\cal L \nsubset \cal L(X(\cal L)) \quad \text{and} \quad \cal L(X(\cal L)) \nsubset \cal L\, .$$
For a non-erasing morphism $\sigma: \cal A^* \to \cal B^*$ and an arbitrary subshift $X \subset \cal A^\Z$ one always has $\sigma(\cal L(X)) \subset \cal L(\sigma(X))$, but this inclusion is in general not an equality. However, since $\alpha_\sigma$ is letter-to-letter, in this special case one has indeed
\begin{equation}
\label{eq3.0}
\alpha_\sigma(\cal L(Z)) \,\, = \,\, \cal L(Y)\, ,
\end{equation}
for $Z:= \pi_\sigma(X)$ and $Y := \sigma(X)$ (which satisfy by definition $\alpha_\sigma(Z) = Y$). Hence we derive from Proposition \ref{3.2}:

\begin{lem}
\label{3.3}
If $\sigma$ is recognizable in $X$, then the complexity functions for $Z = \pi_\sigma(X)$ and $Y = \sigma(X)$ satisfy
\begin{equation}
\label{eq3.1}
p_Y(n + 2r) \geq p_Z(n)
\end{equation}
for any integer $n \geq 0$ and the repetition bound $r \geq 0$ from Proposition \ref{3.2}.
\end{lem}

\begin{proof}
If $w$ and $w'$ are words in $\cal L(Y)$ of length $n + 2r$, then by (\ref{eq3.0}) there are words $u, u' \in \cal L(Z))$ with $\alpha_\sigma(u) = w$ and $\alpha_\sigma(u') = w'$. Since $\alpha_\sigma$ is letter-to-letter, both $u$ and $u'$ have also length $n+2r$. From Definition \ref{3.1} (2) and Proposition \ref{3.2} we know that 
$u\chop_r \neq u'\chop_r\, $ implies $w \neq w'\, $, 
with $|u\chop_r| = | u'\chop_r| = n$.
Hence the number of words of length $n+2r$ in $\cal L(Y)$ must be bigger or equal to the number of words in $\cal L(Z)$ that have length $n$.
\end{proof}

Recall now that for a finite alphabet $\cal C$ any word $w \in \cal C^*$ of length $n \geq 0$ can be prolonged in at most $\card(\cal C)^{m -n}$ different ways to give a word of length $m \geq n$ with $w$ as prefix or with $w$ as suffix. We deduce:
\begin{equation}
\label{eq3.3}
p_Z(n+2r) \,\, \leq \,\, \card(\cal A_\sigma)^{2r} \cdot p_Z(n) \quad \text{for any integer} \quad n \geq 0
\end{equation}

\begin{lem}
\label{3.4}
For any non-erasing morphism $\sigma: \cal A^* \to \cal B^*$ and any subshift $X \subset \cal A^\Z$ with subdivision image $Z = \pi_\sigma(X)$ the complexity functions satisfy:
\begin{equation}
\label{eq3.4}
p_Z(||\sigma|| \cdot n) \,\, \geq \,\, p_X(n)
\end{equation}
\end{lem}

\begin{proof}
Recall that any word in a subshift language can be prolonged arbitrarily (in either direction) to give another word that any also belongs to the same subshift language. It follows that for any integer $n \geq 0$ and any word $w \in \cal L(X) \subset \cal A^*$ of length $|w| = n$ the image word $\pi_\sigma(w)$ can be prolonged (on the right hand side) to give a word $u(w) \in \cal L(Z) \subset\cal A_\sigma^*$ of length $||\sigma|| \cdot n\, $. 
Furthermore, the words $u(w)$ and $u(w')$ are distinct for any words $w \neq w'$ in $\cal L(X)$ of length $n$, as they have different prefixes $\pi_\sigma(w)$ and $\pi_\sigma(w')$ respectively. This proves the claimed inequality.
\end{proof}

We can now prove the statement which is the main goal of this section:

\begin{prop}
\label{3.5}
Let $\cal A\, , \,\, \cal B$ be non-empty finite sets, and let $\sigma: \cal A^* \to \cal B^*$ be a non-erasing monoid morphism. Let $X \subset \cal A^\Z$ be any subshift over $\cal A$, and denote by $Y: = \sigma(X)$ the image subshift.

If $\sigma$ is recognizable in $X$, then there are constants $c > 0$ and $d > 0$ such that the complexity functions $p_X$ and $p_Y$ satisfy 
$$p_Y(m) \,\, \geq \,\, c \cdot p_X(d\cdot m)$$
for any sufficiently large integer $m \geq 0$.
\end{prop}

\begin{proof}
As explained in
Remark \ref{canonical-decomposition} we decompose $\sigma$ canonically into $\sigma = \alpha_\sigma \circ \pi_\sigma$ and we consider the intermediate image subshift $Z := \pi_\sigma(X)$.

We now pick the constant $d$ to satisfy  $0 < d < \frac{1}{||\sigma||}$ 
and observe that for $m \geq \frac{1}{\frac{1}{||\sigma||} - d}$ we have $d\cdot m \leq \frac{1}{||\sigma||} \cdot (m - ||\sigma||)$, so that 
the monotony of the complexity function $p_X$ implies $p_X(d\cdot m) \leq p_X(\frac{1}{||\sigma||} \cdot (m -k))$ for any integer $k \in [1, ||\sigma||]$. The right choice of $k$ assures that $m' := \frac{1}{||\sigma||} \cdot (m - k))$ is an integer, so that the monotony of $p_Z$ and inequality (\ref{eq3.4}) implies $p_Z(m) \geq p_Z(m-k) = p_Z(||\sigma|| \cdot m') \geq p_X(m')\geq p_X(d\cdot m)$ for all sufficiently large integers $m$.

We now set $r \geq 0$ to be the repetition bound on $Z$ for the morphism $\alpha_\sigma$ obtained from Proposition \ref{3.2}, and set the constant $c > 0$ to be equal to $c = \frac{1}{\tiny\card(\cal A_\sigma)^{2r}}$, in order to obtain from (\ref{eq3.3}) and from the last paragraph the inequalities $p_Z(m - 2r) \geq c \cdot p_Z(m) \geq c \cdot p_X(d\cdot m)$ for all sufficiently large integers $m$. We then apply Lemma \ref{3.3} to derive $p_Y(m) \geq p_Z(m - 2r) \geq  c \cdot p_X(d\cdot m)$
for any sufficiently large integer $m \geq 0$.
\end{proof}

\section{The counter-example}
\label{sec:4}

We first observe from the previous section that in the special case where the given morphism $\sigma: \cal A^* \to \cal B^*$ is letter-to-letter, the associated subdivision morphism $\pi_\sigma$ is a monoid-isomorphism, so that we can identify $\sigma$ with the canonically associated letter-to-letter morphism $\alpha_\sigma$. We can hence use the same argument as in the last paragraph of the proof of Proposition \ref{3.5} (with $X$ replacing $Z$ and $\sigma$ replacing $\alpha_\sigma$) to deduce:

\begin{prop}
\label{4.1}
Let $\cal A\, , \,\, \cal B$ be non-empty finite sets, and let $\sigma: \cal A^* \to \cal B^*$ be a monoid morphism which is letter-to-letter. Let $X \subset \cal A^\Z$ be any subshift over $\cal A$, and denote by $Y: = \sigma(X)$ the image subshift.

If $\sigma$ is recognizable in $X\,$,
then for the  constant $c = \frac{1}{\tiny \card(\cal A)^{2r}}$ the complexity functions $p_X$ and $p_Y$ satisfy 
\begin{equation}
\label{eq4.1}
p_Y(m) \,\, \geq \,\, c \cdot p_X(m)
\end{equation}
for any sufficiently large integer $m \geq 0$.
Here $r \geq 0$ is the repetition bound for $\sigma$ in $X$ which is given by Proposition \ref{3.2}.
\qed
\end{prop}

It follows from this proposition, together with the inequalities from Lemma \ref{X.1}, 
that the complexity functions of $X$ and $\sigma(X)$, in case of a letter-to-letter morphism $\sigma$ which is recognizable in $X$, 
must belong to the same $\Theta$-equivalence class. This proves statement (2) of Theorem \ref{1.1}. 
However, the original goal of the author, namely to show that the same statement is true without the assumption ``letter-to-letter'', turns 
out to be impossible to achieve; we will now present a counter-example.

\medskip

Let $\cal A$ be any finite alphabet with at least two letters. Let $\cal A_{\rm II}$ be the ``double'' of $\cal A$ which consists of letters $a_i^-$ and $a_i^+$ for any $a_i \in \cal A$. Let $\sigma_{\rm II}: \cal A^* \to \cal A_{\rm II}^*$ be the {\em subdivision morphism} defined by
$$\sigma_{\rm II}(a_i) = a_i^- a_i^+$$
for all $a_i \in \cal A$. We observe:

\begin{lem}
\label{4.2}
(1) The morphism $\sigma_{\rm II}$ is recognizable in the full shift $\cal A^\Z$ (and hence in any subshift $X \subset \cal A^\Z$).

\smallskip
\noindent
(2)
For any subshift $X \subset \cal A^\Z$ with image subshift $Y := \sigma_{\rm II}(X)$ the complexity functions satisfy
$$p_Y(2n-1) \,\,=\,\, 2 p_X(n)  \quad \text{and}\quad p_Y(2n) \,\,=\,\, p_X(n) + p_X(n+1)$$
for any integer $n \geq 1$.
\end{lem}

\begin{proof}
(1)
The morphism $\sigma_{\rm II}$ is a subdivision morphism, and it is well known that any subdivision morphism is recognizable in the full shift. This can be either seen directly from Definition \ref{3.0a} via ``desubstitution'' of $\sigma_{\rm II}(\cal A^\Z)$, or else from the fact that any subdivision morphism is induced by a subdivision of the graph $\cal R_\cal A$ (a ``rose'') which canonically realizes $\cal A^*$ geometrically, and such a subdivision is a homeomorphism and hence induces a bijection between the sets of biinfinite paths that realize the elements of $\cal A^\Z$ and those of $\sigma_{\rm II}(\cal A^\Z)$.

\smallskip
\noindent
(2)
Any word $w \in \cal L(Y) \subset \cal A_{\rm II}^*$ of odd length $|w| = 2n-1$ must either be the prefix or the suffix of the image 
$\sigma_{\rm II}(u)$ for some word $u \in \cal L(X)$ with $|u| = n$. This proves the first of the two claimed equalities.

Similarly, any word $w \in \cal L(Y) \subset \cal A_{\rm II}^*$ of even length $|w| = 2n$ must either be equal to the 
image 
$\sigma_{\rm II}(u)$ for some word $u \in \cal L(X)$ with $|u| = n$, or else it is equal to $\sigma_{\rm II}(u)\chop_1$ for some word $u \in \cal L(X)$ with $|u| = n+1$. This proves the second of the two equalities.
\end{proof}

It now suffices to consider any subshift $X$ with sufficiently fast growing complexity function $p_X$. For instance, assume $p_X(n) = 
e^{Cn}$ for some constant $C > 1$, as is true for any subshift $X$ of finite type (an ``SFT''). We compute:
$$\frac{p_X(2n-1)}{p_Y(2n-1)} = \frac{p_X(2n-1)}{2 p_X(n)} = 
\frac{e^{C(2n-1)}}{2  e^{Cn}} = \frac{1}{2} e^{C(n-1)}$$
Hence ``$p_X$ is not $\cal O(p_Y)$'', as an analyst would say, and in particular they belong to distinct $\Theta$-growth-equivalence-classes. This proves statement (4) of Theorem \ref{1.1}.

\medskip

We finish this section by placing the proof scheme used above into a slightly more general context, so that it can be readily used in the next section. Recall the definition of $||\sigma||$ and $\langle \sigma \rangle$ from subsection \ref{sec:2}.

\begin{lem}
\label{5.3}
Let $\sigma: \cal A^* \to \cal B^*$ a non-erasing morphism, and let $X \subset \cal A^\Z$ be any subshift over $\cal A$, with image subshift $Y := \sigma(X) \subset \cal B^\Z$. 
If $X$ has exponential complexity function $p_X(n) = e^{Cn}$ for some constant $C > 1$, then one has:
$$\frac{p_X(\langle \sigma \rangle \cdot n)}{p_Y(\langle \sigma \rangle \cdot n)} \geq  
\frac{1}{||\sigma||} e^{C(\langle \sigma \rangle -1) \cdot n - (\langle \sigma \rangle + C-2)} $$
In particular, if all letters $a_i \in \cal A$ have images of length $|\sigma(a_i)| \geq 2$, then one has:
\begin{equation}
\label{eq5.3}
p_X \notin \cal O(p_{Y})
\end{equation}
\end{lem}

\begin{proof}
If $p_X(n) = e^{Cn}$, then one deduces directly from 
Lemma \ref{X.1}:
$$\frac{p_X(\langle \sigma \rangle \cdot (n-1) +1)}{p_Y(\langle \sigma \rangle \cdot (n-1) +1)} \geq \frac{p_X(\langle \sigma \rangle \cdot (n-1) +1)}{||\sigma|| \cdot p_X(n)} = \frac{e^{C\langle \sigma \rangle \cdot (n-1) +1}}{||\sigma|| \cdot e^{Cn}} = \frac{1}{||\sigma||} e^{C(\langle \sigma \rangle -1) \cdot n - (\langle \sigma \rangle + C-2)} $$
\end{proof}

\begin{rem}
\label{5.4}
From the last proof 
we deduce that statement (\ref{eq5.3}) holds also for subshifts $X$ with complexity function $p_X$ that grows slower than exponential:  it suffices 
for instance that $p_X(n) \in \Theta(n^{g(n)})$ for any unbounded function $g: \N \to \N$, 
as long as one has $\langle \sigma \rangle \geq 2$. 
\end{rem}


\section{Entropy and free group laminations}
\label{sec:5}

Recall 
from (\ref{eq1.3}) 
that the topological entropy of any subshift $X \subset \cal A^\Z$ satisfies $h_X = \lim \frac{\log(p_X(n))}{n}$. We now deduce from Theorem \ref{1.1}:

\begin{cor}
\label{5.1}
Let $\cal A$ and $\cal B$ be non-empty finite alphabets, and let $X \subset \cal A^\Z$ be any subshift. Let $\sigma: \cal A^* \to \cal B^*$ be a non-erasing morphism of free monoids, and let $Y:= \sigma(X)$ be the image subshift of $X$. Then the topological entropies $h_X$ and $h_Y$ satisfy the following:
\begin{enumerate}
\item
Without any further hypotheses one has:
$$h_Y \,\, \leq \,\, h_X$$
\item
If $\sigma$ recognizable in $X$ and letter-to-letter, then one has: 
$$h_Y \,\, = \,\, h_X$$
\item
There exist subshifts $X \subset \cal A^\Z$ and morphisms $\sigma$ which are recognizable in $X$ 
(but not letter-to-letter) 
such that one has:
$$h_Y \,\, < \,\, h_X$$
\end{enumerate}
\end{cor}

\begin{proof}
Statements (1) and (2) are 
immediate consequences of the statements (1) and (2) of Theorem \ref{1.1}. For statement (3) we pick any alphabet $\cal A$ with $\card(\cal A) \geq 2$ and consider the full shift $X = \cal A^\Z$, which satisfies $p_X(n) = \card(\cal A)^n$. 
We set $\sigma = \sigma_{\rm II}$ and note that in Lemma \ref{4.2} (1) it has been shown that $\sigma_{\rm II}$ is recognizable in $X$. 
We then deduce from Lemma \ref{4.2} (2) that $h_X = \lim \frac{\log(\tiny\card(\cal A)^n)}{n} = \log (\card(\cal A))$ while $h_Y = \lim \frac{\log(p_Y(2n-1))}{2n-1} = \lim \frac{\log(2 \, \tiny\card(\cal A)^n)}{2n-1} = \frac{1}{2}\log( \card(\cal A))\, $.
\end{proof}

\medskip

We now turn our attention to the free group $F(\cal A)$ over the alphabet $\cal A$ as basis, which contains the free monoid $\cal A^*$, and the 
canonical 
inclusion $\cal A^* \to F(\cal A)$ is a multiplicative morphism. This set-up, however, is rather treacherous, as the free group $F_N$ of finite rank 
$N := \card(\cal A) \geq 2$ has infinitely many distinct bases, and (contrary to what one is used to from free monoids) none of them is preferred over the others, so that 
in any 
``symbolic dynamics approach'' to free groups 
one has to 
seriously 
take 
basis changes into account.

\smallskip

This has led to the basis-free notions of {\em algebraic laminations} and {\em currents} for any free group $F_N$ (see sections 3 and 10 of \cite{BHL1} 
and the references given there). 
Any choice of a basis $\cal A$ of $F_N$ associates canonically to any algebraic lamination $L$ a subshift $X_\cal A(L)$ and to any current $\mu$ on $F_n$ an invariant measure $\hat \mu$ on the subshift $X_\cal A(\supp(\mu))$, so that the choice of $\cal A$ defines canonically a bijection (or rather ``a homeomorphism'') between the set of subshifts 
over the letters of $\cal A$ (and their inverses !) 
and the set of algebraic laminations in $F(\cal A)$. Similarly, the choice of $\cal A$ establishes a bijection between invariant measures on such subshifts and currents on $F(\cal A)$.

\smallskip

It is hence natural to attempt 
carrying over 
to algebraic laminations 
the tools developed for subshifts, 
where the basic obstruction, the existence of inverses in $F(\cal A)$, is overcome by doubling the alphabet $\cal A$ through passing to $\cal A_\pm = 
\cal A \cup \cal A^{-1}$, for $\cal A^{-1} = \{ \alpha_i \mid \alpha_i^{-1} \in \cal A\}$. 
One then represents 
every element of $F(\cal A)$ by the well defined corresponding reduced word in $\cal A_\pm\,$, and an algebraic lamination $L$ is 
given by any subshift $X_\pm = X_\cal A(L)$  that consists of reduced biinfinite words in $\cal A\cup \cal A^{-1}$\,. Here ``reduced'' means that the subshift $X_\pm$ 
must be contained in the SFT in $\cal A_\pm^\Z$ defined by forbidding any $a_i \alpha_i$ or $\alpha_i a_i$ as factor.

\begin{rem}
\label{X.r}
Any subshift $X \subset \cal A^\Z$ determines canonically an algebraic lamination $L$ which in turn defines the subshift $X_\pm := X_\cal A(L) \subset \cal A_\pm^\Z$. In this special case $X_\pm$ is the ``double'' of $X$ in that it consists precisely of any
${\bf x} = \ldots x_{n-1} x_n x_{n+1} \ldots \in X$ together with its ``inverse'' ${\bf \bar x} = \ldots x^{-1}_{n+1} x^{-1}_n x^{-1}_{n-1} \in (\cal A^{-1})^\Z$.
In particular this implies
$$p_{X_\pm}(n) \,\, = \,\, 2 \, p_X(n)$$
for all $n \in \N$.
\end{rem}

For any change of the basis $\cal A$ to another basis $\cal B$ of $F_N$ there is a well known (``Cooper's'') cancellation bound $C(\cal B, \cal A) \geq 0$ such that for any algebraic lamination $L$ we have
\begin{equation}
\label{eq5.1}
w \in \cal L(X_\cal A(L))\,\, \Longrightarrow \,\,  \phi_{\cal B, \cal A}(w) \chop_{C(\cal B, \cal A)} \in \cal L(X_\cal B(L)) \, ,
\end{equation}
where $\phi_{\cal B, \cal A}(w)$ is the reduced word in $\cal B_\pm$ that represents the same element of $F_N$ as the reduced word $w$ in $\cal A_\pm$. One also has
\begin{equation}
\label{eq5.2}
\frac{1}{||\phi_{\cal A, \cal B}||} \cdot |w| \,\, \leq \,\, |\phi_{\cal B, \cal A}(w)| \,\, \leq \,\, ||\phi_{\cal B, \cal A}|| \cdot |w|
\end{equation}
for any reduced word $w$ in $\cal A_\pm$, with 
$$||\phi_{\cal B, \cal A}|| := \max\{|\phi_{\cal B, \cal A}(a_i)|) \mid a_i \in \cal A\} \, .$$

We can now start the last yet missing proof from the Introduction:

\begin{proof}[Proof of Proposition \ref{5.2}]
(1)
Before starting the formal logics of the proof of statement (1), we first need to state the following general observation:

For any integer $K \geq 0$ and any reduced word $u \in \cal A_\pm^*$ the reduced word 
$\phi_{\cal B, \cal A}(u\chop_K) \in \cal B_\pm^*$, which represents the same element of $F_N$ as the chopped word $u\chop_K$, must contain the chopped word 
$\phi_{\cal B, \cal A}(u)\chop_{K \cdot ||\phi_{\cal B, \cal A}||}$ as factor.
It follows that for any integer $K' \geq 0$ the word $\phi_{\cal B, \cal A}(u)\chop_{K \cdot ||\phi_{\cal B, \cal A}|| + K'}$ is a factor $y_{s+1} y_{s+2} \ldots y_t$ of $\phi_{\cal B, \cal A}(u\chop_K)\chop_{K'} =: y_1 y_2 \ldots y_r$.

We also note that $|\phi_{\cal B, \cal A}(u\chop_K)| \leq |\phi_{\cal B, \cal A}(u)| + 2K \cdot ||\phi_{\cal B, \cal A}||\,$, so that 
for sufficiently large $|u|$ 
we have:
$$|\phi_{\cal B, \cal A}(u\chop_K)\chop_{K'}| - |\phi_{\cal B, \cal A}(u)\chop_{K \cdot ||\phi_{\cal B, \cal A}|| + K'}| 
$$
$$\leq |\phi_{\cal B, \cal A}(u)| + 2K \cdot ||\phi_{\cal B, \cal A}||- 2K' - (|\phi_{\cal B, \cal A}(u)| - {2K \cdot ||\phi_{\cal B, \cal A}|| - 2K'})
$$
$$= {4K \cdot ||\phi_{\cal B, \cal A}||}
$$
It follows for the factor $y_{s+1} y_{s+2} \ldots y_t = \phi_{\cal B, \cal A}(u)\chop_{2K \cdot ||\phi_{\cal B, \cal A}|| + K'}$ of $y_1 y_2 \ldots y_r = \phi_{\cal B, \cal A}(u\chop_K)\chop_{K'}$ that
\begin{equation}
\label{eqX.x}
0 \,\, \leq \,\, s \,\,\leq\,\, 4K \cdot ||\phi_{\cal B, \cal A}||\, 
\end{equation}
so that there are at most $4K \cdot ||\phi_{\cal B, \cal A}||+1$ such factors.

\smallskip

In order to prove now statement (1) we first observe from the implication (\ref{eq5.1}) that for any (long) reduced word $w \in \cal L(Y_\pm)$ and its correspondent $u :=\phi_{\cal A, \cal B}(w)$  the chopped word $u' := 
u\chop_{C(\cal A, \cal B)}$ belongs to $\cal L(X_\pm)$. 
By the same argument, the word 
$\phi_{\cal B, \cal A}(u') \chop_{C(\cal B, \cal A)}$ belongs to $\cal L(Y_\pm)$. 

We now apply the above ``general observation'' with $K = C(\cal A, \cal B)$ and $K' = C(\cal B, \cal A)$ to deduce that the word $\phi_{\cal B, \cal A}(u)\chop_{C(\cal A, \cal B) \cdot ||\phi_{\cal B, \cal A}|| + C(\cal B, \cal A)}$ is a factor of $\phi_{\cal B, \cal A}(u\chop_{C(\cal A, \cal B)})\chop_{C(\cal B, \cal A)} 
\in \cal L(Y_\pm)$. But since $u = \phi_{\cal A, \cal B}(w)$, we have $\phi_{\cal B, \cal A}(u) = w$, so that we have now shown:
\begin{equation} 
\label{eqZ.4}
\text{for any $w \in \cal L(Y_\pm)$ the word 
$w\chop_{C(\cal A, \cal B) \cdot ||\phi_{\cal B, \cal A}|| + C(\cal B, \cal A)}$ is a factor of $\phi_{\cal B, \cal A}(u')\chop_{C(\cal B, \cal A)}\,$,}
\end{equation}
where 
$u' \in \cal L(X_\pm)$ with $u' = \phi_{\cal A, \cal B}(w)\chop_{C(\cal A, \cal B)}\,$. 

We now set 
$h := C(\cal A, \cal B) \cdot ||\phi_{\cal B, \cal A}|| + C(\cal B, \cal A)$ 
and
observe $|\phi_{\cal B, \cal A}(u')\chop_{C(\cal B, \cal A)}| \leq ||\phi_{\cal B, \cal A}|| \cdot |u'| \leq ||\phi_{\cal B, \cal A}|| \cdot ||\phi_{\cal A, \cal B}|| \cdot |w| =: m$. We can hence 
apply inequality (\ref{eqX.x}) to deduce from (\ref{eqZ.4}):
\begin{equation} 
\label{eqZ.4b}
\text{for any $w \in \cal L(Y_\pm)$ the word 
$w' := w\chop_h$ 
is a factor of $\phi_{\cal B, \cal A}(u'')$ for 
some $u'' \in \cal L(X_\pm)\,$,} 
\end{equation}
where the word $u''$ has length $|u''| = m$, 
and the factor $w'$ is 
``particular'' 
in that it occurs in a prefix of $\phi_{\cal B, \cal A}(u'')\chop_{C(\cal B, \cal A)}$ of length at most $s: = |w| + 4C(\cal A, \cal B) \cdot ||\phi_{\cal B, \cal A}||\,$, 
so that there are at most $4C(\cal A, \cal B) \cdot ||\phi_{\cal B, \cal A}||+1$ such particular factors.

We now notice that for any integer $n \geq 1$ and any word $w' \in \cal L(Y_\pm)$ of length $|w'| = n$ there is a prolongation $w \in \cal L(Y_\pm)$ of length $|w| = n + 2h$ such that $w' = w\chop_h$. From (\ref{eqZ.4b}) we know that there exists a word $u'' \in \cal L(X_\pm)$ of length $|u''| = m$ such that $w'$ is one of $4C(\cal A, \cal B) \cdot ||\phi_{\cal B, \cal A}||+1$ particular factors of $\phi_{\cal B, \cal A}(u'')$. It follows that for any choice of $u''$ there are at most $4C(\cal A, \cal B) \cdot ||\phi_{\cal B, \cal A}||+1$ distinct possible words $w' \in \cal L(Y_\pm)$ of length $|w'| = n$. This proves
$$p_{Y_\pm}(n) \,\,\leq \,\, (4\, C(\cal A, \cal B)\, ||\phi_{\cal B, \cal A}||+1) \cdot p_{X_\pm}(m)$$
for $m = ||\phi_{\cal B, \cal A}|| \cdot ||\phi_{\cal A, \cal B}|| \cdot (n + 2h)
=
||\phi_{\cal B, \cal A}|| \cdot ||\phi_{\cal A, \cal B}|| \cdot n + 2C(\cal A, \cal B) \cdot ||\phi_{\cal B, \cal A}||^2 \cdot ||\phi_{\cal A, \cal B}|| + 2C(\cal B, \cal A))\,||\phi_{\cal B, \cal A}|| \cdot ||\phi_{\cal A, \cal B}||$.

This shows 
that for sufficiently large $n \in \N$ 
the right inequality of statement (1) holds 
for $C =  4\, C(\cal A, \cal B)\, ||\phi_{\cal B, \cal A}||+1$ and $D = ||\phi_{\cal B, \cal A}|| \cdot ||\phi_{\cal A, \cal B}|| + 1$. The left inequality follows directly from the right one, by the symmetry between the two bases $\cal A$ and $\cal B$ of $\FN$ 
and by the monotony of the complexity function $p_{X_\pm}$.

\smallskip
\noindent
(2)
In order to prove statement (2) we can not quite use the counter-example from section \ref{sec:4}, since the morphism $\sigma_{\rm II}$ used there is not invertible. But a small modification of the same proof idea will do: it suffices 
to consider any morphism
$\sigma: \cal A^* \to \cal B^*$ which is invertible and where $|\sigma(a_i)| \geq 2$ holds for all $a_i \in \cal A$. This is the case for instance for the square (or any higher power) of the celebrated Fibonacci substitution $\sigma_{\rm Fib}$ (given by $a_1 \mapsto a_2 \mapsto a_2 \, a_1$).
The invertibility of $\sigma$ assures that $\cal B$ is a second base for the free group $F(\cal A)$, so that, 
using the notation from Remark \ref{X.r}, the equality 
$\sigma(X) = Y$ implies $X_\pm = X_{\cal A}(L)$ and $Y_\pm = X_{\cal B}(L)$ for the algebraic lamination $L$ determined by $X$. 

In particular we know from from Remark \ref{X.r} that  $p_{X_\pm} = 2 p_X$ and $p_{Y_\pm} = 2 p_Y$.
We can hence apply Lemma \ref{5.3} to obtain 
particular subshifts $X$ and $Y = \sigma(X)$ 
with 
$p_{Y_\pm} \notin \cal O(p_{X_\pm})$ and thus $p_{Y_\pm} \notin \Theta(p_{X_\pm})$.

\smallskip
\noindent
(3)
 In order to prove statement (3) we use 
the same subshifts $X_\pm$ and $Y_\pm$ 
as in 
the above proof of (2) 
(which arose from ``doubling'' the subshifts $X$ and $Y$ given by Lemma \ref{5.3}). 
We readily compute 
$h_{X_\pm} = \lim \frac{\log(p_{X_\pm}(n))}{n} = \lim \frac{\log(2p_{X}(n))}{n} = \lim \frac{\log(2 e^{Cn})}{n} = C$ and 
(using Lemma \ref{X.1}) 
$h_{Y_\pm} = \lim \frac{\log(p_{Y_\pm}(n))}{n} = \lim \frac{\log(2p_Y(n))}{n} = \lim \frac{\log(2p_Y(\langle \sigma \rangle \cdot (n-1) +1))}{\langle \sigma \rangle \cdot (n-1) +1} \leq \lim \frac{\log(2 \cdot || \sigma || \cdot  p_X(n))}{\langle \sigma \rangle \cdot (n-1) +1} = \lim \frac{\log(2 \cdot || \sigma ||\cdot e^{Cn})}{\langle \sigma \rangle \cdot (n-1) +1} 
= 
\frac{C}{\langle \sigma \rangle}$.  
We now use again the above mentioned  Fibonacci substitution $\sigma_{\rm Fib}$ and specify 
$\sigma := \sigma_{\rm Fib}^2$ to ensure $\langle \sigma \rangle \geq 2$,
so that we have
$h_{Y_\pm} < h_{X_\pm}$. 

Hence there is no way to meaningfully use either the basis $\cal A$ nor the basis $\cal B$ of $F_2$ to define the topological entropy, for instance of the algebraic lamination $L$ represented by the ``oriented full shift'' $X_\pm = \cal A^\Z \cup (\cal A^{-1})^\Z$, which is much smaller than (and should not be confused with) the ``algebraic full shift'' 
of all reduced biinfinite words in 
$\cal A_\pm^\Z$; the latter would indeed be invariant under basis change.

It remains to show that $h_{X_\pm} = 0$ implies $h_{Y_\pm} = 0$  : we use statement (1) to compute $h_{Y_\pm} = \lim \frac{\log(p_{Y_\pm}(n))}{n} \leq \lim \frac{\log(C \cdot p_{X_\pm}(D \cdot n))}{n} = D\lim \frac{\log C + \log p_{X_\pm}(D \cdot n))}{D\cdot n} = D \cdot h_{X_\pm} = 0 $.

\smallskip

This completes the proof of all 3 claimed statements. 
\end{proof}

At the occasion of the Dyadisc4 conference in Amiens in July 2021, the author proposed the notion of ``{intrinsic} properties'' of subshifts  (see \cite{Lu-Q21}). The precise definition was at the time purposely left open, but in the mean time the following has stabilized to a version of this concept that seems well applicable in practice:

\begin{defn}
\label{5.5}
(1)
Two subshifts $X \subset \cal A^\Z$ and $Y \subset \cal B^\Z$ are {\em intrinsically equivalent} if there exists a  third subshift $Z \subset \cal C^\Z$ as well as 
non-erasing monoid morphisms $\sigma: \cal C^* \to \cal A^*$ and $\sigma': \cal C^* \to \cal B^*$ which are recognizable in $Z$ and satisfy $X = \sigma(Z)$ and  $Y = \sigma'(Z)$.

\smallskip
\noindent
(2)
A property $\cal P$ of 
subshifts 
is an {\em intrinsic property} 
if for any subshift $X$ which has the property $\cal P$ it follows that 
any subshift $Y$ intrinsically equivalent to $X$ must also 
have $\cal P$.
\end{defn}

It can be shown that the relation stated in Definition \ref{5.5} (1) is indeed transitive and hence defines an equivalence relation on the set of all subshifts over finite alphabets.
In this prospective, we can restate some of the results derived in this note as follows:

\begin{rem}
\label{5.6}
The value of the topological entropy $h_X$ and also the $\Theta$-equivalence class of the complexity function $p_X$ are in general {\em not} intrinsic properties of a given subshift $X$. 

However, the property $h_X = 0$ (or $h_X > 0$) is intrinsic, as is also the weaker equivalence class of $p_X$ defined by 
inequalities as in 
statement (1) of Proposition \ref{5.2}.
\end{rem}

It seems however possible that for some special zero-entropy classes of subshifts $X$, with particular, very slow growing complexity function, the class $\Theta(p_X)$ is after all an intrinsic property of $X$. We'd like to point to \cite{HL} for certain observations that indicate such a possibility.

\medskip

The annoyingly obtrusive combinatorics 
encountered in the above proof of statement (1) of Proposition \ref{5.2}, together with the comparative ``gouleyance'' of the proof of the analogous inequalities in Lemma \ref{X.1} and Proposition \ref{3.5}, inspired the author to the following optimistic quest:

\begin{conjecture}
\label{4.5+}
Let $\cal A$ and $\cal B$ be two bases of a free group $\FN$ of finite rank $n \geq 2$, and let $L$ be an algebraic lamination in $\FN$. Then the two subshifts $X_\pm := X_\cal A(L)$ and $Y_\pm := X_\cal B(L)$ are intrinsically equivalent.
\end{conjecture}

One could actually go a step further, since 
for any non-erasing free monoid morphism $\sigma: \cal A^* \to \cal B^*$ 
the induced endomorphism $\phi_\sigma: F(\cal A) \to F(\cal B)$ 
is in general not injective, even if $\sigma$ is injective, or even if (a strictly stronger assumption !) $\sigma$ is recognizable in the full shift $\cal A^\Z$.

\begin{conjecture}
\label{4.6+}
Let $\cal A$ and $\cal B$ non-empty finite alphabets, and let $\phi: F(\cal A) \to F(\cal B)$ be a (not necessarily injective) homomorphism. Let $L$ be an algebraic lamination in $F(\cal A)$, and assume that $\phi$ is recognizable in $L$ (in the sense of Proposition \ref{3.0b}), so that there is a well defined algebraic image lamination $\phi(L)$ in $F(\cal B)$. Then the two subshifts $X_\pm := X_\cal A(L)$ and $Y_\pm := 
X_\cal B(\phi(L))$ are intrinsically equivalent. 
\end{conjecture}

Here the subshift $Y_\pm \subset \cal B_\pm^\Z$ can be directly derived from the subshift $X_\pm \subset \cal A_\pm^\Z$ via the equality $Y_\pm = 
\cal L(\phi_{\cal B, \cal A}(\cal L(X_\pm))$, where for any reduced word $w \in \cal A_\pm^*$ we denote by $\phi_{\cal B, \cal A}(w) \in \cal B_\pm^*$ the well defined reduced word in $\cal B_\pm$ which represents the $\phi$-image in $F(\cal B)$ of the element in $F(\cal A)$ that is represented by $w$.

\end{document}